\documentclass[12pt]{amsart}

\usepackage{mystyle}

\begin{document}

\title{Further Time Regularity for Non-Local, Fully Non-Linear Parabolic Equations}

\author[H. A. Chang-Lara]{H\'ector A. Chang-Lara}
\address{Department of Mathematics, Columbia University, New York, NY 10027}
\email{changlara@math.columbia.edu}

\author[D. Kriventsov]{Dennis Kriventsov}
\address{Department of Mathematics, Courant Institute of Mathematical Sciences, New York University, New York, NY 10012}
\email{dennisk@cims.nyu.edu}

\begin{abstract}
We establish H\"older estimates for the time derivative of solutions of non-local parabolic equations under mild assumptions for the boundary data. As a consequence we are able to extend the Evans-Krylov estimate for rough kernels to parabolic equations.
\end{abstract}
\subjclass{35B45, 35B65, 35K55, 35R09}
\keywords{Further regularity in time, H\"older estimates, Krylov-Safonov, non-local fully non-linear equations}

\maketitle


\section{Introduction}\label{sec:intro}

Parabolic non-local equations are used to describe a quantity (temperature, pressure, expected value of a game, etc.) that is driven by an average of its values. These could be given by a linear equation $u_t = L_K u$ posed over a space-time domain $\W\times(t_0,t_1]\ss\R^n\times\R$ where,
\begin{align*}
L_K u(x,t) = (2-\s)\int_{\R^n} \left[u(x+y,t)-u(x,t)-y\cdot Du(x,t) 1_{B_{1}}(y)\right]\frac{K(x,y)}{|y|^{n+\s}}dy.
\end{align*}
This is the infinitesimal generator of a purely discontinuous L\'evy process of order $\s\in(0,2)$ with a non-negative L\'evy measure 
$\m_x(dy) = (2-\s)\frac{K(x,y)}{|y|^{n+\s}}dy$. In particular, $K(x,y)=1$ corresponds to a positive multiple of the fractional laplacian, $C_{n,\s}\D^{\s/2} = -C_{n,\s}(-\D)^{\s/2}$ defined in terms of the Fourier multiplier $\widehat{(-\D)^{\s/2}} = |\xi|^{\s}$. The factor $(2-\s)$ is included in order to obtain a positive multiple of the standard laplacian as $\s\to2^-$, connecting in this way with the classical second order theory.

Fully non-linear operators can be obtained, for instance, by considering inf and sup combinations of linear operators as above,
\[
Iu(x,t) = \inf_\a \sup_\b L_{K_{\a,\b}} u(x,t).
\]
In particular, these are the constructions that appear in stochastic games with multiple players. They should be kept in mind whenever we refer to a fully non-linear operator $Iu$ in this introductory section. In Section \ref{sec:preliminaries} we recapitulate the precise definition of fully non-linear uniformly elliptic operators following the reference \cite{MR2494809}.

L. Caffarelli and L. Silvestre studied in a series of papers \cite{MR2494809, MR2781586, MR2831115} the interior regularity of the elliptic problem $Iu = f(x)$. Their approach adapted the Krylov-Safanov and Evans-Krylov theory for fully non-linear equations to the non-local setting. This allowed them to recover uniform estimates as the order $\s\to2^-$, extending the second order theory of fully non-linear equations to non-local problems. In the parabolic setting, the first author of this paper in collaboration with G. D\'avila considered in \cite{MR3148110,MR3115838,2014arXiv1408.5149L} the corresponding parabolic estimates by adapting the strategies from \cite{MR2494809, MR2781586, MR2831115}. We discuss further developments of the theory in the following paragraphs, in particular those closely related with the present work.

In this paper we address the time regularity of the solution. This is one remarkable point of departure between the local and the non-local equations. Let us recall that solution of the local heat equation $u_t = \D u$ over $\W\times(t_0,t_1] \ss \R^n\times\R$ is $C^\8$ in space and time on the interior of the domain regardless of the nature of the initial and boundary data. On the other hand, for the fractional equation $u_t = \D^{\s/2} u$ over $\W\times(t_0,t_1]$ this turns out not to be the case. While solutions will instantly become $C^\8$ in space, a simple example (see \cite{MR3148110}) shows that if the boundary data drastically changes at a given time then $u$ might be no smoother than Lipschitz continuous in time. In other words, the non-local nature of $\D^{\s/2}$ is more sensitive to changes of the data posed over the complement of $\W$.

Our main theorem says that whenever the complementary data is H\"older continuous in time, then $u_t$ is H\"older continuous in space and time with a corresponding estimate. In the following statement $\bar \a\in(0,1)$ is the exponent from the Krylov-Safonov Theorem, a positive constant depending only on the dimension and the ellipticity constants of $I$.

\begin{theorem}\label{thm:intro}
Let $I$ be uniformly elliptic of order $\s\in(0,2)$ with $I0 = 0$ and $u$ a classical solution of,
\begin{alignat*}{2}
u_t - Iu &= f(x,t) \qquad &&\text{ in } \qquad  B_1\times (-1,0],\\
u &= g \qquad &&\text{ on } \qquad  (\R^n\sm B_1)\times (-1,0].
\end{alignat*}
If for some $\gamma \in (0, \bar \a)$, $f(x,\cdot), g(x,\cdot) \in C^{0,\g/\s}(-1,0]$ uniformly in $x$, then for some constant $C>0$ depending on $\min(\gamma,(\bar \a-\gamma))$ we have the a priori estimate,
\[
 \sup_{\substack{x,y\in B_{1/2}\\t,s\in(-1/2,0]}}\frac{|u_t(y,s)-u_t(x,t)|}{(|x-y|+|t-s|^{1/\s})^\gamma} \leq C\1\sup_{x\in B_1} \|f(x,\cdot )\|_{C^{0,\g /\s}(-1,0]} + \sup_{x\in \R^n\sm B_1} \|g(x,\cdot )\|_{C^{0,\g/\s}(-1,0]}\2.
\]
\end{theorem}

We also obtain an almost optimal result in the sense that if the complementary data is just bounded, then $u$ is H\"older continuous in time for every exponent less than one, see Corollary \ref{cor:bdd_rhs}. Whether, in the bounded data case, the solution is actually Lipschitz in time or not remains open.

Let us briefly compare this to known results for non-local equations. When $\s\in(0,1)$, the case of bounded data was treated by J. Serra in \cite{MR3385173} and, same as in our case, proves a H\"older estimate in time for every exponent less than one. When $\s\in(1,2)$, as in our theorem, the best known results assert that (for bounded data) $u$ is H\"older continuous in time for some exponent slightly bigger that $1/\s$ (see \cite{MR3385173,MR3115838}), which is substantially weaker than our result when $\s$ is away from one. For the case of Lipschitz continuous complementary data, an argument using the comparison principle and the Krylov-Safonov estimate gives that $u_t$ is H\"older continuous, which is far from optimal in the dependence on the complementary data. For linear equations with H\"older continuous data, T. Jin and J. Xiong showed in \cite{jinlinear} that $u_t$ is H\"older continuous.

For non-local parabolic equations in particular, differentiability in time can be a very convenient property. For example, for concave equations, analogues of the Evans-Krylov theorem and Schauder estimates are quite difficult, and have only recently been established for nonlocal operators without extra smoothness assumptions on the kernels in \cite{jinnonlinear,serra2014c},  in the elliptic case. In the final section, we give an easy application of our result to show that their theorems still hold for parabolic equations with H\"older continuous data and such kernels.

The first idea in the strategy consists on trading off the boundary data for a right-hand side by truncating the tail of the solution, this is already a standard technique for non-local equations. Our main contribution consists of showing a diminish of oscillation for the incremental quotients $\frac{\d_\t u(x,t)}{\t^\b}:=\frac{u(x,t)-u(x,t-\t)}{\t^\b}$ by assuming some H\"older continuity of the right-hand side introduced by the truncation. This involves several challenges; on one hand, the equation for $\d_\t u/\t^\b$ has a right-hand side that might degenerate as $\t$ approaches zero. On the other hand, by using the corresponding scaling for $\d_\t u/\t^\b$ we make $u$ grow. The key idea is to assume some small a priori H\"older continuity for $\d_\t u/\t^\b$ which gives a way to control the difference quotients for $\t$ arbitrarily small by the difference quotients with $\t$ bounded away from zero. This is rigorously established in the proof of Lemma \ref{lem:diminish_of_oscilation}.

We have recently applied the same argument to a related problem for second-order fully non-linear parabolic equations in \cite{2015arXiv150406294C}. The techniques are very similar and avoid some of the technical difficulties found in the non-local case. A more general result for linear second-order parabolic equations in divergence and non divergence form has been obtained in \cite{2015arXiv150200886D} by different methods.

\subsection{Applying the Main Result}\label{ss:applying}

At face value, Theorem \ref{thm:intro} applies to smooth solutions, but here we outline how to apply it to obtain information about viscosity solutions in several situations. We do not define viscosity solutions here, but rather refer to \cite{MR3148110} or \cite{schwab2014regularity}.

First, consider the initial-boundary value problem
\begin{alignat*}{2}
u_t - Iu &= f(x,t) \qquad &&\text{ in } \qquad  B_1\times (-1,0],\\
u &= g \qquad &&\text{ on } \qquad  (\R^n\sm B_1)\times (-1,0],
\end{alignat*}
where $g$ is a bounded continuous function. A consequence of Theorem \ref{thm:intro} and a certain approximation procedure is that \emph{there exists} a viscosity solution $u$ to this problem which also satisfies the conclusions of the theorem. This approximation procedure is explained in the appendix, Section \ref{ss:approx}.

Second, say that the operator $I$ in the initial-boundary value problem above admits a comparison principle between a viscosity supersolution and a viscosity subsolution. Then the problem has a unique solution, and so (applying the previous observation) \emph{any} solution will inherit the estimates of Theorem \ref{thm:intro}. This is known to hold when $I$ is translation-invariant, i.e. when it commutes with spatial translations. This was established for stationary solutions in \cite{Barles2008567} and \cite{MR2494809}; the parabolic case is essentially identical, with some details given in the appendix of \cite{MR2737806}. Thus for a \emph{translation-invariant} operator $I$, our theorem applies equally well to viscosity solutions. In \cite{Barles2008567}, certain classes of $x$-dependent operators are also shown to admit a comparison principle, and these results may also be applied to parabolic equations.

Third, say that the operator $I$ admits a comparison principle between one viscosity solution and one special solution constructed by approximation (which will inherit our estimate). This leads to the same conclusion as in the previous situation. The most obvious use of this remark is when our special solution turns out to be classical; this is the case when $I$ is convex or concave, and we explain the procedure in detail in Section \ref{sec:applications}. The point is that a comparison principle between a classical solution and a viscosity solution is an immediate consequence of the definition of viscosity solution. There has also been some recent work (see \cite{Mou2015}) on comparison principles between a viscosity solution and another viscosity solution with extra regularity properties, which may be of use in similar arguments.

The paper is organized as follows: Section \ref{sec:preliminaries} explains our notation and gives basic definitions. Section \ref{sec:main} contains the proof of Theorem \ref{thm:intro}. Then in Section \ref{sec:applications} we discuss how to apply the a priori estimate and use it to derive a non-linear parabolic Schauder theorem. The appendix contains some lemmas about H\"older spaces and approximation of viscosity solutions by smooth solutions of perturbed problems.

\section{Preliminaries}\label{sec:preliminaries}

\subsection{Notation}

For functions $q = q(x)$ (typically of just time or just space) we use the notation,
\begin{align*}
[q]_{C^{0,\a}_* (A)} &:=  \sup_{x,y\in A}\frac{|q(x)-q(y)|}{|x-y|^\a},\\
\|q\|_{C^{0,\a}_* (A)} &:=\sup_{x\in A}|q(x)| +[q]_{C^{0,\a}_* (A)}.
\end{align*}
The spaces $C^{k,\a}_*$ are defined in the usual way, demanding that all derivatives of $q$ of order $k$ are in $C^{0,\a}_*$, and the lower-order derivatives are bounded.

Given $\W\subset\R^n, \ A\subset\R^n\times\R$ and $\a,\t\in(0,1)$,
\begin{align*}
\p_p\1\W \times (t_1,t_2]\2 &:= \1\R^n \times \{t_1\}\2 \cup \1(\R^n\sm \W)\times(t_1,t_2]\2,\\
[u]_{C^{0,\a}(A)} &:= \sup_{(x,t),(x',t') \in A} \frac{|u(x,t)-u(x',t')|}{(|x-x'| + |t-t'|^{1/\s})^\a},\\
\|u\|_{C^{0,\a}(A)} &= \sup_A|u| +[u]_{C^{0,\a}(A)}\\
\d_\t u(x,t) &:= u(x,t) - u(x,t-\t),\\
\|u\|_{L^1_\s} &:= \int_{\R^n} |u(y)|\min\11,|y|^{-(n+\s)}\2dy.
\end{align*}
We will frequently use the cylinders $Q_r(x,t) := B_r(x)\times (t-r^\s,t)$. Whenever we omit the center we are assuming that they get centered at the origin in space and time. As it is standard for evolution type problems we consider the parabolic topology on $\R^n\times\R$ generated by neighborhoods of the form $Q_r(x,t)$ with respect to the point $(x,t)$.

To discuss classical solutions, we say that (a function of space only) is in $\text{Class}(B_r(x_0))$ if $u\in C_*^{1,\s-1+\e}(B_r(x_0))$ for some $\e>0$ and $u\in L^1_\s$. A function of space and time belongs to $\text{Class}(Q_r(x_0,t_0))$ if it is continuously differentiable in time, has $u(\cdot,t)\in C_*^{1,\s-1+\e}(B_r(x_0))$ for each $t\in(t_0-r^\s,t_0]$, and lies in $C((t_0-r^\s,t_0]\rightarrow L^1_\s)$.

\subsection{Non-local Uniformly Elliptic Operators}

Given $\s \in (0,2)$, a measurable kernel $K:\R^n\to[0,\8)$ and a vector $b \in \R^n$, the non-local linear operator $L_{K,b}^\s$ is defined by,
\begin{align*}
L_{K,b}^\s u(x) &:= (2-\s)\int \d u(x;y)\frac{K(y)dy}{|y|^{n+\s}} + b\cdot Du(x),\\
\nonumber \d u(x;y) &:= u(x+y) - u(x) - Du(x)\cdot y\chi_{B_1}(y).
\end{align*}

Given that $u\in\text{Class}(\W)$, it is enough that $K$ is bounded for the integral to converge. In order for $L^\s_{K,b}$ to be uniformly elliptic it suffices that $K$ is bounded away from zero (more general conditions on the kernel under which the elliptic theory can be developed had been recently studied in \cite{schwab2014regularity}). Here is the family of linear operators we will be dealing in this paper.

\begin{definition}
For $\s\in(1,2)$ and $0<\l\leq\L<\8$, let $\cL_0 = \cL_0^\s(\l,\L)$ be the family of linear operators $L^\s_{K,b}$ such that
\begin{align*}
K(y)\in[\l,\L] \text{ for all } y\in\R^n \qquad\text{and}\qquad |b|\leq\frac{\L}{\s-1}.
\end{align*}
\end{definition}

This family is scale invariant in the following sense. Consider a rescaling $\tilde u(x)= u(\kappa x)$ for $\k\in(0,1)$, then we have that,
\[
\1L^\s_{\tilde K,\tilde b}\tilde u\2(x) = \k^\s \1L^\s_{K,b}u\2(\k x),
\]
where,
\begin{align*}
\tilde K(y) &:= K(\kappa y) \qquad\text{and}\qquad \tilde b := \kappa^{\s-1}\1b + (2-\s)\int_{B_1\sm B_\k}\frac{yK(y)}{|y|^{n+\s}}dy\2.
\end{align*}
Scale invariance then means that $L^\s_{K,b} \in \cL_0$ implies $L^\s_{\tilde K,\tilde b} \in \cL_0$ as well.

We will use the following extremal operators,
\[
\cM^+ u(x) := \sup_{L\in\cL_0} Lu(x) \qquad\text{and}\qquad \cM^-u(x) := \inf_{L\in\cL_0} Lu(x).
\]
By the scale invariance of $\cL_0$ we get the following homogeneity for $\cM^\pm$: Given a rescaling $\tilde u(x) := u(\k x)$,
\[
\1\cM^{\pm} \tilde u\2(x) = \k^\s \1\cM^\pm u\2(\k x).
\]

Non-linear operators are now obtained as a combination of linear operators.

\begin{definition}\label{def:operator}
A function $I:\W\times \cup_{B_r(x)\ss \W}\text{Class}(B_r(x)) \to \R$ (written as $I(x,u)=Iu(x)$) is called a \emph{non-local operator} of order $\s$. 
We say that $I$ is \emph{uniformly elliptic} if for every $B_r(x) \in \W$, and $u,v\in \text{Class}(B_r(x))$,
\begin{align*}
\cM^-(u-v)(x)\leq Iu(x) -Iv(x) \leq \cM^+(u-v)(x).
\end{align*}
We say that $I$ is $1$-\emph{continuous} if $Iu(\cdot)$ is continuous on $B_r(x)\ss \W$ for every $u\in \text{Class}(B_r(x))$.
We say that $I$ is $\8$-\emph{continuous} if $Iu(\cdot)$ is continuous on $B_r(x)\ss \W$ for every $u\in \text{Class}(B_r(x))\cap L^\8(\R^n)$.
We say that $I$ is \emph{translation-invariant} if it commutes with translation operators, i.e. if $T_h v(x)=v(x-h)$, $B_r(x) \ss \W$, and $|h|<r$, then $T_h Iu(x)=IT_h u(x)$ for $u\in \text{Class}(B_r(x))$.
\end{definition}

Let us summarize some properties and examples of non-local operators:
\begin{enumerate}
\item If $I$ is a uniformly elliptic non-local operator on $\W$, then $Iu(x)$ depends only on the bounded sequence $\{L u(x)\}_{L\in \cL_0}$ and $x$. We may therefore write $Iu(x)=I(x,\{L u(x)\})$ as an operator defined on the product of $\W$ and a subset of the space of sequences $l^\8 (\cL_0)$. The uniform ellipticity of $I$ guarantees that for any pair of sequences $\{a_L\},\{b_L\}$ in $l^\8 (\cL_0)$ it is defined for, we have
 \[
  \inf \{a_L-b_L\} \leq I(x,\{a_L\}) -I(x,\{b_L\}) \leq \sup \{a_L-b_L\}.
 \]
It will be helpful later that (for each $x$) there is a way of extending $I$ to the entire space $l^\8 (\cL_0)$ so that it preserves the above uniform ellipticity property. One way of accomplishing this is via the formula
\[
 I(x,\{a_L\}) = \inf_{\{b_L\}\in U} \1I(x,\{b_L\}) + \sup\{a_L-b_L \}\2,
\]
where $U\ss l^\8 (\cL_0)$ represents those sequences on which $I(x,\cdot)$ is already defined (i.e. ones of the form $\{L v(x)\}$ for $v\in \text{Class}(\W)$).
 \item Any translation-invariant uniformly elliptic non-local operator is $1$-continuous. The extremal operators $\cM^\pm$ are examples of translation-invariant uniformly elliptic non-local operators. More generally, so is any operator of the form
 \[
  Iu = \inf_\a \sup_\b L_{\a,\b}u
 \]
provided the collection $\{L_{\a,\b}\}\ss \cL_0$.
 
 \item An operator of the form
 \[
  Iu = \inf_\a \sup_\b L_{\a,\b,x}u =\inf_\a \sup_\b (2-\s)\int \d u(x;y)\frac{K_{\a,\b}(x,y)dy}{|y|^{n+\s}} + b_{\a,\b}(x)\cdot Du(x)
 \]
where $\{L_{\a,\b,x}\}\ss \cL_0$, is a uniformly elliptic non-local operator. A sufficient condition for this to be $1$-continuous is that the family $\{b_{\a,\b}\}$ is equicontinuous and that $\{K_{\a,\b}(\cdot,y)\}$ is equicontinuous in $\a,\b,$ and $y$; $\8$-continuity only requires the weaker condition
\[
 \| K_{\a,\b}(x+h,\cdot) - K_{\a,\b}(x,\cdot)\|_{L^1_\s}\rightarrow 0
\]
uniformly in $\a,\b$, as $h\rightarrow 0$. The operators we consider in Section \ref{sec:applications} are $\8$-continuous.
\item The notion of being $1$- or $\8$-continuous is a very weak one, and is related to the stability properties of viscosity solutions. As we are proving an a priori estimate, we do not require it; however, to pass to viscosity solutions it will typically be needed. It is typical of works treating $x$-dependent non-local equations to build this kind of assumption into the notion of non-local operator (see, e.g. \cite{MR2781586}).
\end{enumerate}


The following estimate can be found in \cite{chang2014h}.

\begin{theorem}[Krylov-Safonov]\label{thm:KS}
There exists a universal exponent $\bar\a\in (0,1)$ and constant $C$ such that for $u \in \text{Class}(Q_1)$ satisfying in $Q_1$,
\begin{align*}
u_t - \cM^+ u &\leq |f(t)| \quad \text{ and } \quad u_t - \cM^- u \geq -|f(t)|,
\end{align*}
then
\begin{align*}
\|u\|_{C^{0,\bar\a}\1 Q_{1/2}\2} \leq C\1\sup_{t\in(-1,0]}\|u(t)\|_{L^1_\s} +\|f\|_{L^1(-1,0]}\2.
\end{align*}
In particular, given $I$ uniformly elliptic with $I(x,0) = 0$, a similar estimate holds when $u$ satisfies,
\begin{align*}
u_t - Iu = f(x,t) \text{ in $Q_1$}.
\end{align*}
In this case $\|f\|_{L^1(-1,0]}$ has to be replaced with $\|\sup_{x\in B_1}|f(x,\cdot)|\|_{L^1 (-1,0]}$ or just $\sup_{Q_1}|f|$ in the $C^{0,\bar\a}$ estimate.
\end{theorem}

The theorem applies to viscosity solutions as well, but we will not require that here. From now on we fix $\bar \a \in (0,1)$ to be the exponent in the theorem above.

Finally we introduce the following semi-norm in order to measure the regularity of the boundary data in an integrable fashion in space,
\begin{align*}
[g]_{C^{0,\gamma/\sigma}_{\sigma,r}(a,b]} &:= \sup_{(t-\t,t]\ss(a,b]} \left\|\frac{\d_\t g(t)}{\t^\b}\chi_{\R^n \sm B_r}\right\|_{L^1_\s},\\
&= \sup_{(t-\t,t]\ss(a,b]} \int_{\R^n\sm B_r} \frac{|g(y,t)-g(y,t-\t)|}{\t^\b}\min(1,|y|^{-(n+\s)})dy.
\end{align*}
We say that $g \in C^{0,\gamma/\sigma}_{\sigma,r}(a,b]$ if $[g]_{C^{0,\gamma/\sigma}_{\sigma,r}(a,b]}<\8$.

\begin{remark}
The estimate in Theorem \ref{thm:KS} is one of the main tools we will use in Section \ref{sec:main} in order to prove our result. The same strategy for a more general class of operators, as the one considered in the recent paper \cite{schwab2014regularity}, or other current research such as in \cite{2014arXiv1412.7566K,2015arXiv150908320D,2014arXiv1405.4970K}, could be applied as long as a similar estimate, controlled in terms of $\sup_{t\in(-1,0]}\|u(t)\|_{L^1_\s}$, is available. 
\end{remark}

\subsection{Precise Statement of Main Theorem}

We now state the theorem which we will prove in this paper, using the notation introduced above.

\begin{theorem}\label{thm:intro2}
Let $\s\in (1,2)$, $I$ be uniformly elliptic of order $\s$ with $I0 = 0$, and $u \in \text{Class}( Q_2) $ satisfies
\begin{align*}
u_t - Iu &= f(x,t) \text{ classically in } Q_2,\\
u &= g \text{ on } \p_p Q_2.
\end{align*}
Assume that for all $x\in B_2$, $f(x,\cdot) \in C^{0,\gamma/\s}_*[-2^\s,0]$ and $g \in C^{0,\gamma/\sigma}_{\sigma,2}(-2^\s,0]$ for some $\gamma \in (0, \bar \a)$ where $\bar \a$ is the exponent from the Krylov-Safonov theorem. Then $u_t$ exists pointwise, and for some constant $C>0$ depending on $\min(\gamma,(\bar \a-\gamma))$,
\[
 \|u_t\|_{C^{0,\gamma}(Q_{1/2})} \leq C\1\sup_{t\in(-2^\s,0]}\|u(t)\|_{L^1_\s} + [g]_{C^{0,\gamma/\sigma}_{\sigma,2}(-2^\s,0]} + \sup_{x\in B_2}\|f(x,\cdot)\|_{C^{0,\gamma/\s}_*(-2^\s,0]}\2.
\]
Assume instead that $f$ is only bounded. Then for every $\b<1$, there is a constant $C=C(\b)$ such that
\[
\sup_{x\in B_1} [u(x,\cdot)]_{C_*^{0,\b}([-1,0])} \leq C\1\sup_{t\in(-2^\s,0]}\|u(t)\|_{L^1_\s}  + \sup_{Q_2}|f|\2.
\]
The constants depend only on the ellipticity constants of $I$, and in particular remain uniform as $\s\rightarrow 2$.
\end{theorem}

This theorem implies Theorem \ref{thm:intro}.

\section{Proof of the Main Theorem}\label{sec:main}

The goal of this section is to establish Theorem \ref{thm:main} to follow. It includes only the first half in Theorem \ref{thm:intro2}. The case of bounded right-hand side and bounded boundary data is obtained as a preliminary result in Corollary \ref{cor:bdd_rhs}.

\begin{theorem}\label{thm:main}
Let $I$ be uniformly elliptic with $I(x,0) = 0$ and $u \in \text{Class}( Q_2)$ satisfies,
\begin{align*}
u_t - Iu &= f(x,t) \text{ in } Q_2,\\
u &= g \text{ on } \p_p Q_2
\end{align*}
Assume that for all $x\in B_2$, $f(x,\cdot) \in C^{0,\gamma/\s}_*[-2^\s,0]$ and $g \in C^{0,\gamma/\sigma}_{\sigma,2}(-2^\s,0]$ for some $\gamma \in (0, \bar \a)$ where $\bar \a$ is the exponent from the Krylov-Safonov theorem. Then $u_t$ exists pointwise, and for some constant $C>0$ depending on $\min(\gamma,(\bar \a-\gamma))$,
\[
 \|u_t\|_{C^{0,\gamma}(Q_{1/2})} \leq C\1\sup_{t\in(-2^\s,0]}\|u(t)\|_{L^1_\s} + \sup_{x\in B_2}\|f(x,\cdot)\|_{C^{0,\gamma/\s}_*(-2^\s,0]} + [g]_{C^{0,\gamma/\sigma}_{\sigma,2}(-2^\s,0]}\2.
\]
\end{theorem}

The key step is established in the following lemma. Notice that for $\e=0$ the following statement is just a diminish of oscillation leading to a $C^{0,\a}$ estimate for the difference quotient.

\begin{lemma}[Diminish of Oscillation]\label{lem:diminish_of_oscilation}
Let $u \in \text{Class}( Q_2) $ satisfy the following inequalities in $Q_1$ for some $\d\geq0$ and every $\t \in(0,1)$,
\begin{align*}
(\d_\t u)_t - \cM^+ \d_\t u \leq \d \qquad\text{and}\qquad (\d_\t u)_t - \cM^- \d_\t u \geq -\d.
\end{align*}
Given $\b\in(0,1)$, $\a\in(0,\bar\a)$ and $\e \in(0,(\bar\a-\a))$, such that,
\[
\b+\e/\s<1,
\]
there exists constants $\m,\d_0\in(0,1)$ depending on $(\bar\a-\a)$ and $\e$, such that,
\[
\sup_{\substack{\t\in(0,1)\\i\in\N_0}} \mu^{\a i}\left[\frac{\d_{\t} u}{\t^\b}\right]_{C^{0,\e}\1B_{\m^{-i}}\times(-1,0]\2}\leq 1 \qquad \text{ and } \qquad \d\in[0,\d_0],
\]
imply,
\[
\sup_{\t\in(0,1)}\left[\frac{\d_{\t} u}{\t^\b}\right]_{C^{0,\e}(Q_\m)}\leq \m^\a.
\]
\end{lemma}

\begin{proof}
The value $\mu\in(0,1)$ will remain fixed for the duration of the proof; it will be specified later explicitly. Assume by contradiction that there exists $u$ that satisfies the following inequalities in $Q_1$,
\begin{align*}
&(\d_\t u)_t - \cM^+ \d_\t u \leq \d \qquad\text{and}\qquad (\d_\t u)_t - \cM^- \d_\t u \geq -\d.
\end{align*}
Moreover,
\begin{align*}
&\sup_{\substack{\t\in(0,1)\\i\in\N_0}} \mu^{\a i}\left[\frac{\d_{\t} u}{\t^\b}\right]_{C^{0,\e}\1B_{\m^{-i}}\times(-1,0]\2}\leq 1.
\end{align*}
However, there exists a cylinder $Q_{r}(x_0,t_0) \ss Q_\m$ for which,
\[
 \sup_{\t\in (0,1)}\osc_{Q_{r}(x_0,t_0)}\frac{\d_\t u}{\t^\b}>\mu^\a r^\e.
\]

Consider the following rescaling for $\kappa := r/\mu$,
\begin{align*}
&w(x,t) := \kappa^{-(\s\b+\e)}u\1\kappa x+x_0,\kappa^\s t+t_0\2,\\
\Rightarrow\qquad &\frac{\d_\t w}{\t^\b}(x,t) = \k^{-\e}\1\frac{\d_{\k^\s\t} u}{(\k^\s\t)^\b}\2(\k x + x_0, \k^\s t + t_0).
\end{align*}

The hypotheses for the difference quotients of $u$ imply that
\begin{align}
\label{eq:rescaledupbd}
&\sup_{\substack{\t\in(0,\kappa^{-\s})\\i\in\N_0}}\m^{\a i}\left[\frac{\d_\t w}{\t^\b}\right]_{C^{0,\e}\1B_{\m^{-i}}\times(-1,0]\2}\leq \sup_{\substack{\t\in(0,1)\\i\in\N_0}} \mu^{\a i}\left[\frac{\d_{\t} u}{\t^\b}\right]_{C^{0,\e}\1B_{\m^{-i}}\times(-1,0]\2}\leq 1,\\
\label{eq:contradiction}
&\sup_{\t\in(0,\kappa^{-\s})}\osc_{Q_{\mu}}\frac{\d_\t w}{\t^\b}=\kappa^{-\e}\sup_{\t\in(0,1)}\osc_{Q_{r}(x_0,t_0)}\frac{\d_\t u}{\t^\b}>\mu^{\a+\e}.
\end{align}

The next step consists of showing that a hypothesis similar to \eqref{eq:contradiction} holds taking the supremum with respect to $\tau$ away from zero. Namely $\t\in(\bar\t,\kappa^{-\s})$ for some $\bar\t\in(0,\k^{-\s})$ depending on $\m$ and $\e$. Indeed, define for $(x,t),(y,s)\in Q_\m$,
\[
z(a)=w\1x,t+\kappa^{-\s}a\2-w\1y,s+\kappa^{-\s}a\2,
\]
applying Corollary \ref{lem:appendix3} to $z$:
\begin{align*}
&\sup_{\t\in(\bar\t,\kappa^{-\s})} \left|\frac{\d_\t w(x,t)}{\t^{\b}}-\frac{\d_\t w(y,s)}{\t^{\b}}\right|,\\
\geq &\frac{1}{2}\sup_{\t\in(0,\kappa^{-\s})}\left|\frac{\d_\t w(x,t)}{\t^{\b}}-\frac{\d_\t w(y,s)}{\t^{\b}}\right| - C \bar\t^{\e/\s}\sup_{\t\in(0,\kappa^{-\s})}\left[\frac{\d_\t w}{\t^\b} \right]_{C^{0,\e}(Q_1)}.
\end{align*}
The second term on the right-hand side is controlled from \eqref{eq:rescaledupbd}. After taking the supremum in $(x,t),(y,s)\in Q_\m$ and using \eqref{eq:contradiction}, this gives
\begin{align}
\label{eq:contradiction2}
\sup_{\t\in(\bar\t,\kappa^{-\s})} \osc_{Q_\m}\frac{\d_\t w}{\t^\b} \geq \frac{\m^{\a+\e}}{2}-C\bar\t^{\e/\s} \geq \frac{\m^{\a+\e}}{4},
\end{align}
provided that $\bar\t^{\e/\s}$ is sufficiently small with respect to $\m^{\a+\e}$.

Let us fix some $\t\in(\bar\t,\kappa^{-\s})$ and
\[
v(x,t) = \frac{\d_\t w}{\t^\b}(x,t) - \frac{\d_\t w}{\t^\b}(0,0)
\]
By scaling and the homogeneity of the extremal operators we get that $v$ satisfies two inequalities in $Q_1$,
\begin{align*}
 v_t - \cM^+v \leq \k^{\s-\s\b-\e}\frac{\d}{\t^\b} \leq \frac{\d}{\t^\b} \qquad \text{ and } \qquad v_t - \cM^-v \geq -\k^{\s-\s\b-\e}\frac{\d}{\t^\b} \geq -\frac{\d}{\t^\b}
\end{align*}
In order to use the Krylov-Safonov Theorem \ref{thm:KS} we need to control the two terms in the right hand side of such estimate. By applying \eqref{eq:rescaledupbd} on each of the large annuli $B_{\m^{-i}}\sm B_{\m^{-(i-1)}}$, we have the following estimate uniform for the variable $t\in(-1,0]$ which is omitted,
\[
\|v\|_{L_\s^1}\leq \int_{B_1} |v| + \sum_{i=1}^\8 \int_{B_{\m^{-i}}\sm B_{\m^{-(i-1)}}} \frac{|v(y)|}{|y|^{n+\s}}\leq C\sum_{i=0}^\8 \m^{i(\s-\a-\e)} = \frac{C}{1-\m^{\s-\a-\e}} \leq C,
\]
provided that $\m$ is sufficiently small. The right-hand sides get controlled by one provided that $\d < \bar\t^\b =: \d_0$. Then, by the estimate in Theorem \ref{thm:KS} we get the following contradiction to \eqref{eq:contradiction2}, by fixing now $\m$ sufficiently small in terms of $(\bar\a-(\a+\e))$,
\[
 \osc_{Q_\m} v \leq C\m^{\bar\a}\leq \frac{\m^{\a+\e}}{8}.
\]
\end{proof}

In the following step we iterate Lemma \ref{lem:diminish_of_oscilation} at smaller scales. In this sense, we might take advantage of the modulus of continuity of the right-hand side.

\begin{corollary}[Iteration]\label{cor:iteration}
Let $u \in \text{Class}( Q_2)$ satisfies the following inequalities in $Q_1$ for some $\d\geq0$, $\gamma\in[0,1]$ and every $\t \in(0,1)$,
\begin{align*}
(\d_\t u)_t - \cM^+ \d_\t u \leq \d\t^{\gamma/\s} \qquad\text{and}\qquad (\d_\t u)_t - \cM^- \d_\t u \geq -\d\t^{\gamma/\s}.
\end{align*}
Given $\b\in(0,1)$, $\a\in(0,\bar\a)$ and $\e\in(0,(\bar\a-\a))$ such that,
\[
1+\gamma/\s > \b+\a/\s+\e/\s,
\]
there exists constants $\m,\d_0\in(0,1)$ depending on $(\bar\a-\a)$ and $\e$, such that,
\[
\sup_{\substack{\t\in(0,1)\\i\in\N_0}} \mu^{\a i}\left[\frac{\d_{\t} u}{\t^\b}\right]_{C^{0,\e}\1B_{\m^{-i}}\times(-1,0]\2}\leq 1 \qquad \text{ and } \qquad \d \in[0,\d_0],
\]
imply,
\begin{align*}
\sup_{\substack{i\in\N\\Q_{\m^i}(x_0,t_0)\ss Q_{1/2}\\\t\in(0,\m^{\s i})}} \left[\frac{\d_\t u}{\m^{\a i}\t^\b}\right]_{C^{0,\e}(Q_{\m^i}(x_0,t_0))} \leq 8.
\end{align*}
\end{corollary}

\begin{proof}
Fix $(x_0,t_0) \in Q_{1/2}$ and let,
\[
v(x,t) := u(x/2+x_0,t/2^\s+t_0) \qquad\Rightarrow\qquad \d_\t v(x,t) = \d_{\t/2^\s}u(x/2+x_0,t/2^\s+t_0).
\]
Notice that $v \in C(\bar Q_2) \cap C((-2^\s,0]\to L^1_\s)$ satisfies the following inequalities in $Q_1$ for every $\t \in(0,1)$,
\begin{align*}
&(\d_\t v)_t - \cM^+ \d_\t v \leq \d(\t/2^\s)^{\gamma/\s}(1/2^\s) \leq \d\t^{\gamma/\s},\\
&(\d_\t v)_t - \cM^- \d_\t v \geq -\d(\t/2^\s)^{\gamma/\s}(1/2^\s) \geq -\d\t^{\gamma/\s}.
\end{align*}
Moreover, given that $(x_0,t_0) \in Q_{1/2}$ we get the inclusions $B_{\m^{-i}/2}(x_0) \ss B_{\m^{-i}}$ and $(t_0-1/2^\s,t_0] \ss (-1,0]$ such that,
\[
\sup_{\substack{\t\in(0,1)\\i\in\N_0}} \mu^{\a i}\left[\frac{\d_{\t} v}{\t^\b}\right]_{C^{0,\e}\1B_{\m^{-i}}\times(-1,0]\2} = \frac{1}{2^{\s\b+\e}}\sup_{\substack{\t\in(0,1/2^\s)\\i\in\N_0}} \mu^{\a i}\left[\frac{\d_{\t} u}{\t^\b}\right]_{C^{0,\e}\1B_{\m^{-i}/2(x_0)}\times(t_0-1/2^\s,t_0]\2} \leq 1.
\]
To prove the corollary it suffices to show that,
\begin{align*}
\sup_{\substack{i\in\N\\\t\in(0,2^\s\m^{\s i})}} \m^{-\a i}\left[\frac{\d_\t v}{\t^\b}\right]_{C^{0,\e}(Q_{\m^i})} \leq 1.
\end{align*}

We proceed by induction where the case $i=1$ is already established by Lemma \ref{lem:diminish_of_oscilation} taking $\m<1/2$ if necessary. Assume for some $i\in\N$ and $Q_{\m^i}(x_0,t_0)\ss Q_2$ the inductive hypothesis,
\[
\sup_{\substack{\t\in(0,\min(1,2^\s\m^{\s(i-j)}))\\j\in\N_0}} \m^{\a(j-i)}\left[\frac{\d_{\t} v}{\t^\b}\right]_{C^{0,\e}\1 Q_{\m^{i-j}}\2}\leq 1.
\]

Let
\begin{align*}
&w(x,t) := \frac{v(\m^i x,\m^{\s i}t)}{\m^{\s i(\b+\a/\s+\e/\s)}} \qquad\Rightarrow\qquad \d_\t w(x,t) = \frac{(\d_{\m^{\s i}\t} v)(\m^i x,\m^{\s i}t)}{\m^{\s i(\b+\a/\s+\e/\s)}}
\end{align*}
such that it satisfies the following inequalities in $Q_1$ for every $\t \in(0,1)$,
\begin{align*}
&(\d_\t w)_t - \cM^+ \d_\t w \leq \frac{\d(\m^{\s i}\t)^{\gamma/\s}}{\m^{\s i(\b+\a/\s+\e/\s)}}\m^{\s i} \leq \d,\\
&(\d_\t w)_t - \cM^- \d_\t w \geq -\frac{\d(\m^{\s i}\t)^{\gamma/\s}}{\m^{\s i(\b+\a/\s+\e/\s)}}\m^{\s i} \geq -\d,
\end{align*}
given that $\b+\a/\s+\e/\s<1+\gamma/\s$.

From the inductive hypothesis,
\[
\sup_{\substack{\t\in(0,1)\\j\in\N_0}} \m^{\a j}\left[\frac{\d_{\t} w}{\t^\b}\right]_{C^{0,\e}(Q_1)} \leq \sup_{\substack{\t\in(0,\min(1,2^\s\m^{\s(i-j)}))\\j\in\N_0}} \m^{\a(j-i)}\left[\frac{\d_{\t} v}{\t^\b}\right]_{C^{0,\e}\1 Q_{\m^{i-j}}\2} \leq 1.
\]
By applying Lemma \ref{lem:diminish_of_oscilation} to $w$ we now obtain,
\begin{align*}
1 &\geq \sup_{\t\in(0,1)} \m^{-\a i}\left[\frac{\d_\t w}{\t^\b}\right]_{C^{0,\e}(Q_\m)} = \sup_{\t\in(0,\m^{\s i})} \m^{-\a(i+1)}\left[\frac{\d_\t v}{\t^\b}\right]_{C^{0,\e}(Q_{\m^{i+1}})},\\
&\geq \sup_{\t\in(0,2^\s\m^{\s(i+1)})} \m^{-\a(i+1)}\left[\frac{\d_\t v}{\t^\b}\right]_{C^{0,\e}(Q_{\m^{i+1}})},
\end{align*}
which shows the desired inductive step and concludes the proof the corollary.
\end{proof}

The next corollary establishes an estimate over a higher order difference quotient by sacrificing a little bit of the $\e$ H\"older exponent.

\begin{corollary}[Improvement of the Difference Quotient]\label{cor:improvement_dif_quot}
Let $u \in \text{Class}( Q_2)$ satisfies the following inequalities in $Q_1$ for some $\d\geq0$, $\gamma\in[0,1]$ and every $\t \in(0,1)$,
\begin{align*}
(\d_\t u)_t - \cM^+ \d_\t u \leq \d\t^{\gamma/\s} \qquad\text{and}\qquad (\d_\t u)_t - \cM^- \d_\t u \geq -\d\t^{\gamma/\s}.
\end{align*}
Given $\b\in(0,1)$, $\a\in(0,\bar\a)$ and $\e\in(0,(\bar\a-\a))$ such that,
\[
1+\gamma/\s > \b+\a/\s+\e/\s,
\]
there exists constants $\m,\d_0\in(0,1)$ depending on $\a$ and $\e$, such that,
\[
\|u\|_{C^{0,\bar\a}(Q_1)} + \sup_{\substack{\t\in(0,1)\\i\in\N_0}} \m^{\a i}\left[\frac{\d_{\t} u}{\t^\b}\right]_{C^{0,\e}\1B_{\m^{-i}}\times(-1,0]\2}\leq 1 \qquad \text{ and }\qquad \d< \d_0,
\]
imply the following estimates:
\begin{itemize}
\item If $\b+\a/\s+\e/\s<1$ then,
\[
\sup_{\t\in(0,1/4^\s)}\left[\frac{\d_{\t} u}{\t^{\b+\a/\s}}\right]_{C^{0,\bar\a\e/4}\1 Q_{1/4}\2}\leq C,
\]
For some constant $C>0$ depending on $(1-(\b+\a/\s+\e/\s))$.
\item If $\b+\a/\s+\e/\s>1$ then,
\[
\|u_t\|_{C^{0,\s\b+\a+\e-\s}(Q_{1/4})} \leq C,
\]
For some constant depending $C>0$ on $((\b+\a/\s+\e/\s)-1)$.
\end{itemize}
\end{corollary}

\begin{proof}
By Corollary \ref{cor:iteration} we know that there exists $C>0$ such that,
\begin{align*}
\sup_{\substack{Q_r(x,t)\ss Q_{1/2}\\\t\in(0,r^\s)}}\osc_{Q_r(x,t)}\frac{\d_\t u}{r^{\a+\e}\t^\b}\leq C.
\end{align*}
Let $x\in B_{1/2}$ and $v(t)=u(x,t)$. From the above estimate using $\t=r^\s$,
\begin{align}\label{eq:13}
\left|\frac{\d^2_\t v(t)}{\t^{\b+\a/\s+\e/\s}}\right| \leq \osc_{Q_{\t^{1/\s}}(x,t)} \frac{\d_\t u}{\t^{\b+\a/\s+\e/\s}} \leq C.
\end{align}

Let us consider the case $\b+\a/\s+\e/\s<1$. Here our goal is to get the estimate,
\[
\sup_{\substack{Q_r(x,t) \ss Q_{1/4}\\\t\in(0,1/4^\s)}} \osc_{Q_r(x,t)}\frac{\d_\t u}{r^{\bar\a\e/4}\t^{\b+\a/\s}} \leq C.
\]
At this point we notice that \eqref{eq:13} and $\osc_{(-1,0]}v \leq 1$ are the main hypothesis used in the proof of Lemma 5.6 from \cite{Caffarelli95} in order to obtain the following estimate,
\begin{align}\label{eq:14}
\sup_{(t-\t,t]\ss(-1/2^\s,0]}\frac{\d_\t v(t)}{\t^{\b+\a/\s+\e/\s}} \leq C \qquad\underset{\text{triangle inequality}}{\Rightarrow}\qquad \sup_{\t\in(0,4^{-\s})}\osc_{Q_{1/4}} \frac{\d_\t u}{\t^{\b+\a/\s+\e/\s}} \leq C.
\end{align}
We now fix $Q_r(x,t)\ss Q_{1/4}$ and consider two cases. If $\t\in(0,r^{\bar\a/\s})$ then from the previous estimate,
\begin{align*}
\osc_{Q_r(x,t)}\frac{\d_\t u}{r^{\bar\a\e/4}\t^{\b+\a/\s}} \leq C\frac{\t^{\e/\s}}{r^{\bar\a\e/4}}\leq C.
\end{align*}
If $\t\in[r^{\bar\a/\s},1/4^{\s})$ then we use $\|u\|_{C^{0,\bar\a}(Q_1)}\leq 1$ and the fact that $1>\b+\a/\s+\e/\s$,
\begin{align*}
\osc_{Q_r(x,t)}\frac{\d_\t u}{r^{\bar\a\e/4}\t^{\b+\a/\s}} \leq 2\frac{r^{\bar\a(1-\e/4)}}{\t^{\b+\a/\s}} \leq 2.
\end{align*}
This concludes the estimate in the case $\b+\a/\s+\e/\s<1$,

Let us consider now the case $\b+\a/\s+\e/\s>1$. From \eqref{eq:13} and using Lemma \ref{lem:appendix4} instead we get the bounds
\begin{equation}\label{eq:15}
 \sup_{Q_{1/2}}|u_t| +  \sup_{\substack{(x,t)\in Q_{1/2}\\\t\in(0,1/2^\s)}} \left| \frac{\d_\t u_t(x,t)}{\t^{\b+\a/\s+\e/\s-1}}   \right| \leq C.
\end{equation}
At this moment all we have to show is that given $(y,s)\in Q_r(x,t)\ss Q_{1/4}$,
\begin{align}\label{eq:7}
 \left| u_t(x,t) - u_t(y,s)\right| \leq Cr^{\s\b+\a+\e-\s}.
\end{align}
Let $\t = r^\s$ such that from \eqref{eq:15} we obtain,
\[
 \left| \frac{\d_{r^\s}u(x,t)}{r^\s} - \frac{\d_{r^\s}u(y,s)}{r^\s}\right| \leq Cr^{\s\b+\a+\e-\s}.
\]
On the other hand, using \eqref{eq:15} once again,
\[
\left|\frac{\d_{r^{\s}}u}{r^\s}-u_t\right|(y,s) \leq \frac{1}{r^\s}\int_{-r^\s}^0|u_t(y,s+a)-u_t(y,s)|da\leq Cr^{\s\b+\a+\e-\s}.
\]
A similar bound also occurs if we replace $(y,s)$ by $(x,t)$. Finally the desired estimate results from the triangle inequality by adding and subtracting $\1\frac{\d_{r^\s}u(x,t)}{r^\s} - \frac{\d_{r^\s}u(y,s)}{r^\s}\2$ inside the absolute value in \eqref{eq:7}.
\end{proof}

At this point we can give the proof of Theorem \ref{thm:intro} in the case of bounded right-hand side and integrable tails.

\begin{corollary}[Bounded Right-Hand Side]\label{cor:bdd_rhs}
Let $I$ be uniformly elliptic with $I(x,0) = 0$ and $u \in \text{Class}( Q_2)$ satisfies,
\begin{align}\label{eq:16}
u_t - Iu &= f(x,t) \text{ in } Q_2.
\end{align}
Given $\b\in(0,1)$ there exists $\e\in(0,1-\b)$ such that,
\[
\sup_{\t\in(0,r_0^\s)} \left[\frac{\d_\t u}{\t^\b}\right]_{C^\e\1Q_{r_0}\2} \leq C\1\sup_{t\in(-2^\s,0]}\|u(t)\|_{L^1_\s}+\sup_{Q_2}|f|\2.
\]
For some $r_0\in(0,1)$ universal and $C>0$ depending on $(1-\b)$. 
\end{corollary}

\begin{proof}
Assume without loss of generality,
\begin{align}\label{eq:17}
\sup_{t\in(-2^\s,0]}\|u(t)\|_{L^1_\s}+\sup_{Q_2}|f| \leq 1.
\end{align}

Let
\begin{align*}
&\b_0 \in(0,\bar\a/4),\quad \a:=\bar\a/2, \quad \b_k := \b_0 + k\a/\s, \quad r_k := 1/16^{k+1}.
\end{align*}
Our goal is to prove that as long as $\b_N<1$,
\begin{align}\label{eq:18}
\sup_{\t\in (0,r_N^\s)} \left[\frac{\d_\t u}{\t^{\b_N}}\right]_{C^{0,\e_N}\1Q_{r_N}\2} \leq C_N,
\end{align}
where
\[
\e_N := \min((1-\b_N)/4,\bar\a/4)(\bar\a/4)^N.
\]
Then the result follows by a standard covering argument both for the domain of the equation and the interval of H\"older exponents. Notice also that $\b_N<1$ implies $N<4/\bar\a$ such that any dependence on $N$ is actually universal. On the other hand, the dependence on $\e_N$ degenerates as $\b_N$ approaches 1.

The idea is to iterate Corollary \ref{cor:improvement_dif_quot} using $\gamma=0$. In order to do this we consider at each step the following truncation starting with $v_0 = u$,
\[
v_k(x,t) := \eta(x) v_{k-1}(x/16,t/16^\s),
\]
where $\eta(x) \in [0,1]$ is a smooth function supported in $B_4$ and equal to one in $B_2$. Our goal is to establish the following inductive steps:
\begin{enumerate}[label=(IH\arabic*)]
\item\label{1} For $k\in\N_0$ and every $\t\in(0,1)$ the following inequalities get satisfied in $Q_1$,
\begin{align*}
(\d_\t v_k)_t - \cM^+ \d_\t v_k \leq -C_{1,k} \qquad\text{and}\qquad (\d_\t v_{k})_t - \cM^- \d_\t v_{k} \geq -C_{1,k}.
\end{align*}
\item\label{2} For $k=1,2,\ldots,(N+1)$,
\[
\sup_{\substack{\t\in(0,1)\\i\in\N_0}}\m^{\a i} \left[\frac{\d_\t v_k}{\t^{\b_{k-1}}}\right]_{C^{0,\e_{k-1,N}}\1B_{\m^{-i}}\times(-1,0]\2} \leq C_{2,k} \]
where
\[
\e_{k,N} := \min((1-\b_N)/4,\bar\a/4)(\bar\a/4)^k.
\]
\end{enumerate}

The hypothesis \ref{1} is satisfied in the case $k=0$ with $C_{1,0} = 2$ from \eqref{eq:16}, the translation invariance of $I$ and the bound for $f$ provided by \eqref{eq:17}. Assuming that \ref{1} holds for an arbitrary $k\in\N_0$ we get that for every $\t\in(0,1)$ the following inequality gets satisfied by $\d_\t v_{k+1}$ in $Q_1$,
\begin{align*}
\underbrace{(\d_\t v_{k+1})_t}_{\substack{= (1/16^\s)(\d_{\t/16^\s} v_k)_t\\ \Leftarrow \eta=1}} -  \cM^+ \d_\t v_{k+1} &\leq (1/16^\s)\1\1\d_{\t/16^\s} v_k\2_t- \cM^+ \d_{\t/16^\s} v_k + \cM^+(\underbrace{(1-\eta) \d_{\t/16^\s} v_k}_{=0 \text{ in } B_1})\2,\\
&\leq C_{1,k} + C\sup_{t\in(-2^\s,0]}\|v_{k}(t)\|_{L^1_\s}.
\end{align*}
If $k=0$, then we use 
\[
\sup_{t\in(-2^\s,0]}\|v_{0}(t)\|_{L^1_\s} = \sup_{t\in(-2^\s,0]}\|u(t)\|_{L^1_\s} \leq 1.
\]
Otherwise, if $k\geq1$ then we use that,
\[
\sup_{t\in(-2^\s,0]}\|v_{k}(t)\|_{L^1_\s} \leq C\|u\|_{L^\8\1Q_{2/16^{k}}\2} \leq C\sup_{t\in(-2^\s,0]}\|u(t)\|_{L^1_\s} \leq C.
\]
In any case, and using a similar argument for the super solution inequality, we get that \ref{1} holds for $k+1$ with $C_{1,k+1} = C_{1,k}+C$.

To obtain \ref{2} for $k=1$ we proceed as in the proof of Corollary \ref{cor:improvement_dif_quot} by considering two cases. If $\t\in(0,r^\s)$ then we bound the oscillation in the time variable in terms of $\t$ by the Krylov-Safonov estimate and then use the triangle inequality to obtain,
\begin{align*}
\sup_{Q_r(x,t)\ss Q_{1/2}} \osc_{Q_r(x,t)} \frac{\d_\t u}{r^{\e_{0,N}} \t^{\b_0}} \leq C\frac{\t^{\bar\a/2-\b_0}}{r^{\e_{0,N}}} \leq C.
\end{align*}
If $\t\in[r^\s,1/2^{\s})$ then we use instead that $\osc_{Q_r(x,t)} \d_\t u \leq \osc_{Q_r(x,t)} u + \osc_{Q_r(x,t-\t)} u$, for which each term gets controlled in terms of $r$, once again using the Krylov-Safonov estimate,
\begin{align*}
\sup_{Q_r(x,t)\ss Q_{1/2}} \osc_{Q_r(x,t)} \frac{\d_\t u}{r^{\e_{0,N}} \t^{\b_0}} \leq C\frac{r^{\bar\a-\e_{0,N}}}{\t^{\b_0}} \leq C.
\end{align*}
Then we get the following estimate for $u$,
\[
\sup_{\t \in (0,1/2^\s)} \left[\frac{\d_\t u}{\t^{\b_0}}\right]_{C^{0,\e_{0,N}}(Q_{1/2})} \leq C.
\]
This establishes \ref{2} for $v_1$ after considering the scaling and the truncation given by $\eta$.

At this point we assume the inductive hypotheses \ref{1} and \ref{2} for some $k\in\{1,2,\ldots,N\}$ and see how to obtain \ref{2} for $k+1$. By the Krylov-Safonov Theorem \ref{thm:KS} and \eqref{eq:17},
\[
\|v_k\|_{C^{0,\bar\a}(Q_1)} \leq \|u\|_{C^{0,\bar\a}(Q_1)} \leq C.
\]
The hypothesis of Corollary \ref{cor:iteration} with $\gamma=0$ now apply to the following function,
\[
\tilde v_k = \frac{v_k}{\d_0^{-1}C_{1,k}+C_{2,k}+C}.
\]
Then, as long as $k\leq N$, such that $1>\b_N$ implies $1>\b_{k-1}+\a/\s + \e_{k-1,N}/\s$, we get that,
\[
\sup_{\t \in (0,1/4^\s)}\left[\frac{\d_\t\tilde v_k}{\t^{\b_{k}}}\right]_{C^{0,\e_{k,N}}} \leq C \qquad\Rightarrow\qquad \sup_{\t \in (0,1/4^\s)}\left[\frac{\d_\t v_k}{\t^{\b_{k}}}\right]_{C^{0,\e_{k,N}}} \leq C(\d_0^{-1}C_{1,k}+C_{2,k}+1).
\]
This establishes \ref{2} for $v_{k+1}$ with $C_{2,k+1}=C(\d_0^{-1}C_{1,k}+C_{2,k}+1)$.

The final iteration in this inductive argument establishes the desired estimate \eqref{eq:18} for $u$ and concludes the proof of the Corollary.
\end{proof}

To conclude the proof of Theorem \ref{thm:main} we just need to iterate the procedure one more time starting with a H\"older exponent sufficiently close to one. This is possible because of the H\"older hypotheses.

\begin{proof}[Proof of Theorem \ref{thm:main}]
Let us assume without loss of generality that,
\[
\sup_{t\in(-2^\s,0]}\|u(t)\|_{L^1_\s} + \sup_{x\in B_2}\|f(x,\cdot)\|_{C^{0,\gamma/\s}_*(-2^\s,0]} + [g]_{C^{0,\gamma/\sigma}_{\sigma,2}(-2^\s,0]} \leq 1.
\]

By Corollary \ref{cor:bdd_rhs} we know that for $\b := 1-\frac{\bar\a-\gamma}{4\s}$ there exists some $\e\in(0,1-\b)$ such that,
\[
\sup_{\t\in(0,1/2^\s)} \left[\frac{\d_\t u}{\t^\b}\right]_{C^{0,\e}(Q_{1/2})} \leq C.
\]

Same as in the proof of Corollary \ref{cor:bdd_rhs}, we consider the truncation,
\[
v(x,t) :=\eta(x)u(x/16,t/16^\s).
\]
By using the H\"older hypothesis for the right-hand we get that the following inequality is satisfied in $Q_1$ for every $\t\in(0,1)$,
\begin{align}\label{eq:11}
(\d_\t v)_t - \cM^+ \d_\t v \leq C\t^{\gamma/\s} + (1/16^\s)\cM^+((1-\eta)\d_{\t/16^\s} u).
\end{align}
Now we split the second term in the following way using $\tilde \t = \t/16^\s$ and $\tilde t = t/16^\s$,
\begin{align}\label{eq:10}
\cM^+((1-\eta)\d_{\tilde \t} u) &\leq C\1\|\d_{\tilde \t} u\1\tilde t\2\|_{L^1(B_2)} + \|\d_{\tilde \t} g\chi_{\R^n\sm B_2}\|_{L^1_\s}\2,\\
\nonumber
&\leq C\1\|\d_{\tilde \t} u\1\tilde t\2\|_{L^1(B_2)} + \t^{\gamma/\s}\2.
\end{align}
In order to control $\|\d_{\tilde\t} u\1\tilde t\2\|_{L^1(B_2)}$ we consider $Q_{2d}(x_0,t_0) \ss Q_2$ and $v(x,t) := u(dx+x_0,d^\s t+t_0)$. By applying Corollary \ref{cor:bdd_rhs} to $v$ with $\b = \gamma/\s$ we obtain that,
\begin{align}
\nonumber
\frac{\d_\t^2 u}{\t^{\gamma/\s}}(x_0,t_0) \leq Cd^{-\gamma/\s} &\quad \ \Rightarrow\qquad \sup_{(t-2\t]\ss(-1/16^\s,0]}\left\|\frac{\d_\t^2 u(t)}{\t^{\gamma/\s}}\right\|_{L^1(B_2)} \leq \frac{C}{\bar\a-\gamma},\\
\label{eq:9}
&\overset{\text{Lemma \ref{lem:appendix5}}}{\Rightarrow}\quad \sup_{(t-2\t]\ss(-1,0]}\left\|\d_\t u(t)\right\|_{L^1(B_2)} \leq \frac{C \t^{\gamma/\s}}{\bar\a-\gamma}.
\end{align}
Putting \eqref{eq:11}, \eqref{eq:10} and \eqref{eq:9} and we get,
\begin{align}\label{eq:12}
(\d_\t v)_t - \cM^+ \d_\t v \leq C\t^{\gamma/\s} \qquad\text{and}\qquad (\d_\t v)_t - \cM^- \d_\t v \geq -C\t^{\gamma/\s},
\end{align}
where the super solution inequality follows by a similar argument.

Finally, by applying Corollary \ref{cor:iteration} to $v$ in the case $\b+\a/\s+\e/\s>1$ we obtain the desired estimate.
\end{proof}

\section{Applications}\label{sec:applications}

Theorem \ref{thm:intro} may be applied to viscosity solutions of non-local fully non-linear equations. Indeed, combined with the approximation procedure sketched in Proposition \ref{prop:apriori} in the appendix it always guarantees the existence of a viscosity solution $u$ to a Dirichlet problem which satisfies the conclusion of the theorem. This is especially useful if the solutions are known to be unique; this is true whenever $I$ is independent of $x$ (see \cite{MR2494809}), or when there is at least one classical solution. The recent work \cite{Mou2015} gives some additional situations in which uniqueness is known in the elliptic case, and similar results should hold in the parabolic setting. 

The following is an application to the parabolic Evans-Krylov theorem. We note that the result below is equally 
We note that the assumptions on the regularity of $f$, $I$, and the boundary data may not be greatly relaxed unless further restrictions are placed on the class of kernels.

\begin{theorem}Let $\s \in (1,2)$, $u$ be a classical solution of $u_t - Iu = f$ in $Q_2$. Assume that $I$ is a concave uniformly elliptic operator of the form
\[
 Iu(x)= \inf_{\a\in \mathcal{A}} (2-\s)\int \frac{\d u(x;y)K_\a(x,y)dy}{|y|^{n+\s}},
\]
where $K_\a(x,y)=K_\a(x,-y)$. Assume that for some $\a<\bar \a$ with $\s+\a\neq 2$, we have
 \[
 \int_{B_{2r}\sm B_r}|K(x,y)-K(x',y)|dy\leq r^n |x-x'|^\a
 \]
for each $r$,
\[
 \sup_{t\in(-2^{\s},0]}\|u(\cdot,t)\|_{C^{0,\a}_*(\R^n)} + \|u\|_{C^{0,\a/\s}_{\s,2}}\leq 1,
\]
and
\[
 \|f\|_{C^{0,\a}(Q_2)}\leq 1.
\]
Then $u$ admits the estimate
\begin{equation}\label{eq:EKc}
 \|u_t\|_{C^{0,\a}(Q_1)} + \|(-\D)^{\s/2}u\|_{C^{0,\a}(Q_1)}\leq C.
\end{equation}
\end{theorem}

\begin{proof} Applying Theorem \ref{thm:main} to $u$, we have immediately that
 \[
  \|u_t\|_{C^{0,\a}(Q_{5/3})}\leq C.
 \]
 Fixing a time $t\in (-(5/3)^{\s},0]$, we may rewrite the equation as
 \[
   Iu(x)= u_t(x) -f(x) \qquad  x\in B_{5/3},
 \]
with the right-hand side bounded in $C^{0,\a}_*(B_{5/3})$. Applying the theorem of J. Serra \cite{serra2014c}, we obtain the estimate
\begin{equation}\label{eq:EKi1}
 \|u(\cdot,t)\|_{C^{0,\s-1+\a}_*(B_{3/2})}\leq C,
\end{equation}
and this is valid for each $t\in (-(5/3)^\s,0]$. It therefore suffices to show that
\begin{equation}\label{eq:EKi3}
 |(-\D)^{\s/ 2}\d_\t u(x,t)|\leq C\t^{\a/\s}
\end{equation}
for each $t\in (-1,0]$, $x\in B_1$, and $\t<\frac{1}{10}$. We will show instead that
\[
 |(-\D)^{\s / 2}\d^2_\t u(x,t)|\leq C\t^{\a/\s},
\] 
which implies \eqref{eq:EKi3} after applying the proof of Lemma 5.6 in \cite{Caffarelli95}. Below, we will abuse notation by writing 
 \[
  (-\D)^{\s/2} v := (2-\s) \int \frac{\d u(x;y)dy}{|y|^{n+\s}}, 
 \]
as, properly speaking, the fractional Laplacian should have a normalization constant $c(\s,n)$. As $c(\s,n)$ is comparable to $2-\s$ in the range $\s\in (1,2)$, bounding this operator is equivalent.

Assume $\s+\a<2$. First, set
\[
 (-\D)_h^{\s/2}v(x) = (2-\s)\int_{|y|>h}\frac{v(x+y)+v(x-y)-2v(x)}{|y|^{n+\s}}dy. 
\]
Then for any $(x,t)\in Q_{3/2}$, we have that from \eqref{eq:EKi1},
\begin{align}
\nonumber |(-\D)_h^{\s/2}u(x,t)-(-\D)^{\s/2}u(x,t)| &\leq c_\s\sup_s \int_{|y|<h}\frac{|u(x+y,s)+u(x-y,s)-2u(x,s)|}{|y|^{n+\s}}dy,\\
\label{eq:EKi2} &\leq Cc_\s h^\a
\end{align}
for $h<\frac{1}{10}$. This allows us to estimate
\[
 |(-\D)^{\s/2}\d^2_\t u(x,t)|\leq |(-\D)_h^{\s/2}\d^2_\t u(x,t)| + C(2-\s) h^\a.
\]

We estimate the first term further by breaking it into two pieces, and using the H\"older assumption on $u$ (as in \eqref{eq:9}) on the outer one.
\begin{align*}
 |(-\D)_h^{\s/2}\d^2_\t u(x,t)|&\leq C(2-\s)\1 \int_{|y|>\frac{1}{10}}\frac{\d^2_\t u(x+y)+\d^2_\t u(x-y)-2\d^2_\t u(x)}{|y|^{n+\s}}dy+\int_{h<|y|<\frac{1}{10}}\2 \\
 &\leq C(2-\s)\1 \t^{\a/\s} + \int_{h<|y|<\frac{1}{10}}\frac{\d^2_\t u(x+y)+\d^2_\t u(x-y)-2\d^2_\t u(x)}{|y|^{n+\s}}dy\2.
\end{align*}
In the remaining integral term, the numerator is bounded by $C\t^{1+\a/\s}$ from our estimate on $u_t$. This gives
\[
 \int_{h<|y|<\frac{1}{10}}\frac{|\d^2_\t u(x+y)+\d^2_\t u(x-y)-2\d^2_\t u(x)|}{|y|^{n+\s}}dy\leq C(2-\s)\t^{1+\a/\s}h^{-\s},
\]
from computing the integral of the kernel. After setting $h=\t^{1/\s}$ and putting everything together, we obtain
\[
  |(-\D)^{\s/2}\d^2_\t u(x,t)|\leq C\t^{\a/\s},
\]
which implies the conclusion. Note that we used that $2-\s\leq 2$ in each term to remove any dependence on $\s$.

Now assume that $\s+\a>2$. In this case, we first show an estimate on the Laplacian of $u$. We claim that
\begin{equation}\label{eq:EKi4}
 \|\D u\|_{C^{0,\s+\a-2}(Q_{4/3})}\leq C.
\end{equation}
Again, the estimate in space follows directly from \eqref{eq:EKi1}, so it is enough to show that
\[
 |\d_\t^2 \D u|\leq C\t^{\frac{\s+\a-2}{\s}}.
\]
The space estimate implies that
\[
 \left|\D u(x)- \frac{c(n)}{h^2}\fint_{B_h}u(x,t)- u(x+y,t)dy\right|\leq C|h|^{\s+\a-2}.
\]
On the other hand, from the time estimate on $u$,
\[
 \left|\frac{c(n)}{h^2}\fint_{B_h} \d_\t^2 u(x+y,t)-\d^2_\t u(x,t)dy \right|\leq C\t^{1+\a/\s}h^{-2}.
\]
Combining these and setting $h=\t^{1/\s}$ gives that
\[
 |\d_\t^2 \D u|\leq C\t^{\frac{\s+\a-2}{\s}},
\]
and this concludes the proof of \eqref{eq:EKi4}.

Now we proceed exactly as in the case of $\s+\a<2$, except that we offer a different estimate in place of \eqref{eq:EKi2}, which is no longer valid. At any point $(x,t)$, we have the estimate
\[
 |\d_\t^2 u(x+y,t) +\d_\t^2 u(x-y,t) -2\d_\t^2 u(x,t) - y^T D^2 \d_\t^2 u(x,t)y|\leq C |y|^{\s+\a}
\]
from the space estimate \eqref{eq:EKi1}. We apply this to the integral of the incremental quotients along spherical shells, noting that the odd terms cancel:
\[
 \left|\int_{|y|=r}\d [\d_\t^2 u] (x,t;y)dy - c(n) r^{n+1} \D \d^2_\t u(x,t)\right| \leq C r^{n-1+\s+\a}.
\]
When combined with the estimate in time on $\D u$, this implies
\[
 \left|\int_{|y|=r}\d [\d_\t^2 u] (x,t;y)dy\right|\leq C(r^{n-1+\s+\a} + r^{n+1} \t^{\frac{\s+\a-2}{\s}}).
\]
Integrating this in $r$ against the kernel, we obtain that
\begin{align*}
|(-\D)^{\s/2}\d_\t^2 u(x,t) - (-\D)^{\s/2}_h \d_\t^2 u(x,t)|&\leq   (2-\s)\left|\int_{|y|\leq h}\frac{\d [\d_\t^2 u](x,t;y)dy}{|y|^{n+\s}}\right|\\
&\leq C\1 (2-\s)h^{\a} + h^{2-\s}\t^{\frac{\s+\a-2}{\s}}\2.
\end{align*}

Now proceeding as in the case of $\s+\a<2$, this leads to
\[
 |(-\D)^{\s/2}\d_\t^2u(x,t)|\leq C(2-\s)\1 \frac{\t^{1+\a/\s}}{h^\s}+ h^\a\2 + C h^{2-\s}\t^{\frac{\s+\a-2}{\s}}.
\]
Setting $h=\t^{1/\s}$ yields the conclusion.
\end{proof}

\section{Appendix}

In the first part of this appendix we establish a few interpolation results about H\"older spaces. The second half establishes an approximation procedure for viscosity solutions by classical solutions.

The following lemma can be understood as a maximum principle.

\begin{lemma}
For any $u \in C([-1,0])$,
\begin{align*}
&u(-1) = u(0) = 0 \qquad\text{ and }\qquad \sup_{\t\in(0,1)} \left\|\d^2_\t u\right\|_{L^\8[-1+\t,0]} \leq 1 \qquad \Rightarrow\qquad \|u\|_{L^\8[-1,0]} \leq 1.
\end{align*}
\end{lemma}

\begin{proof}
Assume by contradiction and without loss of generality that there exists $t \in [-1/2,0]$ that realizes the strictly positive maximum of $(u-1)$ in $[-1,0]$. Then we obtain the following contradiction,
\[
-1 \leq \d_t^2 u(t) = (u(2t) - u(t)) - (u(t)-u(0)) < 0 - 1.
\]
\end{proof}

%

The proof of Lemma 5.6 in \cite{Caffarelli95} shows that if $\a+\b<1$ then there exits some constant $C>0$ depending on $1-(\a+\b)$ such that the following estimate holds,
\[
\sup_{\t\in(0,1)}\left|\frac{\d_\t u(0)}{\t^{\a+\b}}\right| \leq C\1\osc_{[-1,0]}u + \sup_{\t\in(0,1)}\left[\frac{\d_{\t}u}{\t^{\b}}\right]_{C^{0,\a}_*([-1+\t,0])}\2.
\]

By applying this result followed by the maximum principle to,
\[
\bar u(s) := \frac{u(\bar\t s) + s u(-\bar\t) - (s+1)u(0)}{\bar\t^{\a+\b}\sup_{\t\in(0,\bar\t)}\left[\frac{\d_\t u}{\t^\b}\right]_{C^{0,\a}_*([\t-\bar\t,0])}}
\]
we get the following corollary.

\begin{corollary}\label{lem:appendix3}
Let $\a,\b\in(0,1)$ such that $\a+\b<1$. There exists a constant $C>0$ depending on $1-(\a+\b)$ such that for any $u \in C([-1,0])$ and $\bar\t\in(0,1)$,
\[
\sup_{\t\in(0,\bar\t)}\left|\frac{\d_\t u(0)}{\t^{\b+\a}}\right| \leq \left|\frac{\d_{\bar\t} u(0)}{\bar\t^{\b+\a}}\right| + C\sup_{\t\in(0,1)}\left[\frac{\d_\t u}{\t^\b}\right]_{C^{0,\a}_*([-1+\t,0])}.
\]
In particular,
\[
\sup_{\t\in(\bar\t,1)}\left|\frac{\d_{\t} u(0)}{\t^{\b}}\right| \geq \frac{1}{2}\sup_{\t\in(0,1)}\left|\frac{\d_\t u(0)}{\t^{\b}}\right| - C\bar\t^\a\sup_{\t\in(0,1)}\left[\frac{\d_\t u}{\t^\b}\right]_{C^{0,\a}_*([-1+\t,0])}.
\]
\end{corollary}

Finally, this last lemma establishes a H\"older estimate for the derivative when $\a+\b>1$.

\begin{lemma}\label{lem:appendix4}
Let $\a,\b\in(0,1)$ such that $\b+\a>1$ and $u:[-1,1]\to\R$ such that,
\[
\osc_{[-1,1]} u +\sup_{(t-\t,t]\ss(-1,1]}\frac{\d_\t^2 u(t)}{\t^{\b+\a}} \leq 1.
\]
Then for some constant $C$ depending on $\b+\a-1$,
\[
\|u_t\|_{C^{0,\a+\b-1}_*((-1,1))} \leq C.
\]
\end{lemma}

\begin{proof}
By Lemma 5.6 from \cite{Caffarelli95} we know that $u$ is Lipschitz and therefore differentiable almost everywhere. By a density argument it suffices to show that that for each point of differentiability $t_0\in(-1,1)$,
\begin{align*}
|u(t)-u(t_0)-u_t(t_0)(t-t_0)| \leq C|t|^{\a+\b} \text{ for } t \in [-1,1].
\end{align*}

Assume without loss of generality that $t_0 = 0$, $u(t_0) = 0$ and $u_t(t_0) = 0$. If there exists $h\in(0,1]$ such that $u(h) > Ch^{\a+\b}$, then by iterating the hypothesis of the Lemma we get for every $i\in\N$,
\begin{align*}
\frac{u(2^{-i}h)}{2^{-i}h} > \1C-\sum_{j=0}^{i-1}2^{-(\a+\b-1)j}\2 h^{\a+\b-1} \geq \frac{C}{2}h^{\a+\b-1}>0,
\end{align*}
provided that $C = 4/(2^{\a+\b-1}-1)$. This contradicts $u_t(0)=0$ as $i\to\8$. 
\end{proof}

The following Lemma can be proved as Lemma 5.6 in \cite{Caffarelli95},

\begin{lemma}\label{lem:appendix5}
For $u \in C((-1,0]\to L^1(B_1))$ and $\b\in(0,1)$,
\[
\sup_{(t-\t,t]\ss(-1,0]} \left\|\frac{\d_\t u(t)}{\t^\b}\right\|_{L^1(B_1)} \leq C\1\sup_{t\in(-1,0]}\|u(t)\|_{L^1(B_1)} + \sup_{(t-2\t,t]\ss(-1,0]} \left\|\frac{\d_\t^2 u(t)}{\t^\b}\right\|_{L^1(B_1)}\2
\]
\end{lemma}

\subsection{Approximation by Classical Solutions}\label{ss:approx}

Let $I$ be a uniformly elliptic non-local operator as in Definition \ref{def:operator}. As in (1) of the comments after the definition, we will write here $Iu(x)=I(x,\{L_{K,b}u(x)\}_{L_{K,b}\in \cL_0})$, and for each $x$ extend $I(x,\cdot)$ to an operator defined for any sequence $\{a_L\}$ in $l^\8 (\cL_0)$ satisfying the ellipticity property
\[
 \inf\{a_L-b_L\} \leq I(x,\{a_L\})-I(x,\{b_L\}) \leq \sup \{a_L-b_L\}.
\]
Let us record here that, if $T_h v(x) = v(x-h)$ is the translation operator, we have that $IT_h u(x) = I (x,\{L_{K,b}u(x-h)\})$, while $T_h Iu(x) = I(x-h,\{L_{K,b}u(x-h)\})$. 
Given $L_{K,b}\in \cL_0$ with kernel $K$ and drift $b$, i.e.
\[
 L_{K,b}u(x,t) = (2-\s) \int \frac{\d u (x;y) K(y)}{|y|^{n+\s}}dy + b \cdot  D u,
\]
define $J(y)$ in the following way: first, find the value $R_0\in (0,\8)$ such that
\[
 (2-\s)R_0^{-1} \int_{B_{R_0}} \frac{|y\cdot e|^2 dy}{|y|^{n+\s}} =\frac{4\L}{\l(\s-1)},
\]
where $e$ is a unit vector. This is always possible, as the left-hand side is proportional to $R_0^{1-\s}$. Then set
\[
 J(y) =  y \cdot b \chi_{B_{R_0}}\frac{\l(\s-1)}{4\L R_0}.
\]
This implies that
\[
 b = (2-\s) \int \frac{y J(y)dy}{|y|^{n+\s}},
\]
and also that $J(y)\leq \frac{\l}{4}$. Define $K'(y) = K(y)- J(y),$ which satisfies $K'\in [3\l/4,\L+\l/4]$. We may then rewrite the operator $L_{K,b}$ as
\[
 L_{K,b}= (2-\s)\int \frac{\d u(x;y) K'(y)dy}{|y|^{n+\s}}  - (2-\s) \int \frac{(u(x+y)-u(x-y)) J(y)dy}{|y|^{n+\s}}.
\]
The important point is that we have replaced the local drift with a non-local drift term. Set
\[
 \begin{cases}
  K'_\e (y) = \l\chi_{\{|y|\leq \e\}} + K'(y) \chi_{\{|y|>\e\}} \\
  J_\e (y) = J(y) \chi_{\{|y|>\e\}},
 \end{cases}
\]
and define $L^\e_{K,b}$ by
\[
 L^\e_{K,b}= (2-\s)\int \frac{\d u(x;y) K'_\e(y)dy}{|y|^{n+\s}}  - (2-\s) \int \frac{(u(x+y)-u(x-y)) J_\e(y)dy}{|y|^{n+\s}}.
\]
Note that this may alternatively be rewritten as
\[
 L^\e_{K,b}= (2-\s)\int \frac{\d u(x;y) (K'_\e(y)-J_\e)dy}{|y|^{n+\s}}  + (2-\s)\n u\cdot  \int \frac{y J_\e(y)dy}{|y|^{n+\s}},
\]
which is uniformly elliptic with constants $\l/2,\L+\l/2$.

Finally, choose a smooth positive function $\phi$ supported on $B_1$ with
\[
 \int \phi = 1.
\]
Set $\phi_r (x)=r^{-n} \phi(r x)$. At this point let us define,
\begin{equation}\label{eq:approx}
 I_\e u (x,t) = \int \phi_\e (z-x) I\1z,\{L^\e_{K,b} u(x,t)\}\2dz.
\end{equation}

The following proposition summarizes the useful properties of  $I_\e$.

\begin{proposition}\label{prop:apriori} Let $\W$ be smooth domain, $I$ be an $\8$-continuous uniformly elliptic non-local operator, and let $I_\e$ be as given in \eqref{eq:approx}. Then we have the following:
\begin{enumerate}
 \item $I_\e$ is a uniformly elliptic non-local operator.
 \item Let $v$ be a function in $\text{Class}(Q_r(x,t))$, with $B_r(x)\ss \W$. Assume furthermore that either $I$ is $1$-continuous or that $v$ is globally bounded. Then
 \[
  \sup_{Q_{r/2} (x,t)}|I_\e v - I v| \rightarrow 0 
 \]
as $\e \searrow 0$.
  \item Let $f_\e\in L^\8(\W\times (-1,0])$ and $g_\e \in C([-1,0]\rightarrow L^1_\s)$ be smooth functions. Then the Dirichlet problem
  \begin{equation}
   \begin{cases}
    v_t - I_\e v = f_\e & \text{ on } \W \times (-1,0]\\
    v=g_\e & \text{ on } \p_p (\W\times (-1,0]) 
   \end{cases}
   \end{equation}
  admits a unique viscosity solution $u_\e$ in $C(\W\times [-1,0])\cap C((-1,0]\rightarrow L^1_\s)$. Moreover, $u_\e$ is smooth ($C^{\s+\a}$) on $\W\times (-1,0]$.
  \item Let $f\in L^\8(\W\times (-1,0])$ and $g \in C((-1,0]\rightarrow L^1_\s)\cap C(\W\times \{-1\})$, and take $f_\e$, $g_\e$ as in the above with $f_\e \rightarrow f$ locally uniformly and $\|g_\e-g\|_{C((-1,0]\rightarrow L^1_\s)\cap C(\W\times\{-1\})} \rightarrow 0$. Assume furthermore either that $I$ is $1$-continuous or that $g,g_\e$ are uniformly bounded and $\|g-g_\e\|_{L^\8(\R^n\times (-1,0])}\rightarrow 0$. Then there exists a subsequence $\e_k$ such that $u_{\e_k}$ converges locally uniformly on $\W \cap [-1,0]$ to a function $u$. This $u$ is a viscosity solution to the Dirichlet problem
  \[
   \begin{cases}
    u_t - I u = f & \text{ on } \W \times (-1,0]\\
    u=g & \text{ on } \p_p (\W\times (-1,0]).
   \end{cases}
  \]
\end{enumerate}
\end{proposition}

\begin{proof}[Sketch of proof.] (1) follows from the definition of uniform ellipticity of $I$ and the fact that an average of uniformly elliptic operators is still uniformly elliptic. Indeed, the same argument shows that for given smooth functions $u,v$,
 \[
  I_\e u-I_\e v\leq \sup_{L_{K,b}\in \cL_0} L^\e_{K,b}(u-v),
 \]
a fact which will be useful momentarily.

For (2),  write
\begin{align*}
 I_\e v(y,s) -Iv(y,s) &= \int \phi_\e (z-y) [I\1z,\{L^\e_{K,b} v(y,s)\}\2-I\1z,\{L_{K,b} v(y,s)\}\2]dz \\
 &\qquad +\int \phi_\e (z-y) [I\1z,\{L_{K,b} v(y,s)\}\2 -I\1y,\{L_{K,b} v(y,s)\}\2 ]dz
\end{align*}
In the first term, from the uniform ellipticity of $I$, we have
\[
 |I\1z,\{L^\e_{K,b} v(y,s)\}\2-I\1z,\{L_{K,b} v(y,s)\}\2| \leq \sup \{|L^{\e}_{K,b}v(y,s)-L_{K,b}v(y,s)|\},
\]
which goes to $0$ uniformly in $(y,s)$ for any $v\in \text{Class}(Q_r(x,t))$ by the explicit construction of $L^\e$. For the second term, we may further subdivide it as
\begin{align*}
 \int &\phi_\e (z-y) [I\1z,\{L_{K,b} v(y,s)\}\2 -I\1z,\{L_{K,b} v(z,s)\}\2 ]dz\\
 &+\int \phi_\e (z-y) [I\1z,\{L_{K,b} v(z,s)\}\2 -I\1y,\{L_{K,b} v(y,s)\}\2 ]dz
\end{align*}
The second term in this we recognize as the familiar quantity
\[
 \int \phi_\e (z-y) [Iv(z,s) -Iv(y,s)]dz,
\]
which goes to $0$ because of the continuity assumption on $I$ and the assumption on $v$ (which together guarantee that $Iv$ is a continuous function). For the final remaining quantity, we again use the uniform ellipticity of $I$, this time to say that
\[
 \sup_{|z-y|<\e}|\1z,\{L_{K,b} v(y,s)\}\2 -I\1z,\{L_{K,b} v(z,s)\}\2 |\leq \sup_{L_{K,b},|z-y|<\e} \{ |L_{K,b} v(y,s) - L_{K,b} v(z,s)|\},
\]
which tends to $0$ uniformly for $v\in \text{Class}(Q_r(x,t))$.

We focus on the proof of (3). To this end, it is convenient to write
\[
 I_\e u = \l(2-\s)\int \frac{\d u (x;y) dy}{|y|^{n+\s}} +F(u):=\cL(u)+F(u),
\]
where $F$ is a non-linear non-local operator. We give two estimates on $F$ (see \cite{MR3169771} for complete details), the first of which follows from the uniform ellipticity of $I_\e$:
\[
 F(u)-F(v)=I_\e u -I_\e v -\cL(u-v) \leq \sup_{L_{K,b}\in \cL_0} L^\e_{K,b}(u-v) -\cL(u-v).
\]
Using the form of each $L^\e_{K,b}$, this is bounded by
\[
 C \int_{B_\e^c}\frac{|\d (u-v)(x;y)|dy}{|y|^{n+\s}} \leq C_\e \|u-v\|_{L^1_\s}.
\]
For the second estimate, let $\t_h u(x) =u(x-h)$. Then from the definition of $I_\e$ by convolution, we have
\[
 |\t_{-h}F \t_h u -F u|\leq C_\e |h| \|u\|_{L^1_\s}.
\]

Let $\cB=C^{0,\a}(\R^n\times [-1,0])$ for $\a<\s/2$ fixed, and define the operator $G:\cB\rightarrow \cB$ which maps a function $u$ to the unique viscosity solution to the linear fractional parabolic equation
\[
 \begin{cases}
  G(u)_t - \cL G(u)=f_\e + F(u) & \text{ on } \W\times (-1.0]\\
  G(u)=g_\e & \text{ on } \p_p(\W\times (-1.0]).
 \end{cases}
\]
We claim that $G$ has a fixed point. From the estimates on $F$, it follows that $F:\cB\rightarrow \cB$ is continuous. As a consequence of regularity for the fractional parabolic problem, we have that
\[
 \|G(u)-G(v)\|_{C^{0,\b}(\R^n\times [-1,0])}\leq C\|u-v\|_{\cB},
\]
for some $\b>\a$, so $G$ is continuous and compact. Also, we have that $G(u)$ satisfies an interior $C^{\s+\a}$ estimate, so it is the unique classical solution. If for some $\k\in [0,1]$  we have $u=\k G(u)$, then $u$ satisfies
\begin{equation}\label{eq:prop1}
\begin{cases}
  u_\t - (1-\k)\cL(u) - \k I_\e u  = \k f_\e  &  \text{ on } \W\times (-1.0].\\
  u = \k g_\e & \text{ on } \p_p(\W\times (-1.0]).
\end{cases}  
\end{equation}
The operator $(1-\k)\cL + \k I_\e$ is uniformly elliptic, so from \cite{chang2014h} we have
\[
 \|u\|_{\cB}\leq C,
\]
with the constant independent of $\k$. This implies the existence of a fixed point $u$ from the Leray-Schauder fixed point theorem \cite{gilbarg2001elliptic}. This $u$ is the unique classical solution to \eqref{eq:prop1} (with $\k=1$).

Finally, to prove (4), observe that from the barriers in \cite{chang2014h} we obtain that $\{u_\e\}$ is equicontinuous (depending on the original initial data) in domains of the form $\W'\times[-1,0]$ for $\W'\cc\W$. We may therefore extract a subsequence $u_\e \rightarrow u$ locally uniformly, and $u=g$ on $\W\times \{-1\}$. Applying (2) and a standard argument about viscosity solutions, we obtain that $Iu=f$ on $\W\times [-1,0]$; the details may be found in \cite{MR3115838}.
\end{proof}

\textbf{Acknowledgments:} DK was partially supported by National Science Foundation grant DMS-1065926.

\bibliographystyle{plain}
\bibliography{mybibliography}

\end{document}